\documentclass[12pt]{amsart}
\usepackage{amsmath}
\usepackage{amssymb}
\usepackage{bm}
\usepackage{graphicx}
\usepackage{verbatim}
\newtheorem{Theorem}{Theorem}
\newtheorem{Lemma}{Lemma}
\newtheorem{Proposition}{Proposition}
\newtheorem{Corollary}{Corollary}

\newtheorem{Conjecture}{Conjecture}

\title{Eisenstein Quasimodes and QUE}
\author{Shimon Brooks}

\begin{document}
\maketitle

{\em Abstract:}  We consider the question of Quantum Unique Ergodicity for quasimodes on surfaces of constant negative curvature, and conjecture the order of quasimodes that should satisfy QUE.  We then show that this conjecture holds for Eisenstein series on $SL(2,\mathbb{Z})\backslash\mathbb{H}$, extending results of Luo-Sarnak and Jakobson.  Moreover, we observe that the equidistribution results of Luo-Sarnak and Jakobson extend to quasimodes of much weaker order--- for which QUE is known to fail on compact surfaces--- though in this scenario the total mass of the limit measures will decrease.  We interpret this stronger equidistribution property in the context of arithmetic QUE, in light of recent joint work with E. Lindenstrauss \cite{jointQmodes} on joint quasimodes.

\section{Introduction}\label{Intro}

Let  $M=\Gamma\backslash \mathbb{H}$ be a surface of constant negative curvature, and consider the geodesic flow on the unit cotangent bundle $S^*M$.  It is well known that this dynamical system is ``chaotic"--- eg., mixing (with respect to Liouville measure on $S^*M$), Anosov, etc.  On the other hand, the high-energy spectral data for such surfaces --- which is the semiclassical quantum analog of the geodesic flow--- is extremely mysterious (see eg. \cite{SarnakHyp}).  The Quantum Unique Ergodicity (QUE) Conjecture \cite{RS} states that high energy eigenfunctions should become equidistributed as the eigenvalue tends to $\infty$ , corresponding to the semiclassical limit $\hbar\to 0$; precisely, that the microlocal lifts
$$\mu_\phi : f\in C^\infty(S^*M) \mapsto \langle Op(f)\phi, \phi\rangle$$
converge in the weak-* topology to the Liouville measure on $S^*M$ as the eigenvalue of $\phi$ tends to $\infty$.  Any weak-* limit point of the $\mu_\phi$ is called a {\bf quantum limit}.  

It is known \cite{Shn, Z2, CdV} that any quantum limit is a (positive) measure invariant under the geodesic flow; if we normalize $||\phi||_2=1$, then they are probability measures.  It is further shown by \v{S}nirel'man-Zelditch-Colin de Verdi\'ere that {\em almost all} $\mu_\phi$ become equidistributed, in the following sense:  for any choice of orthonormal basis of $L^2(M)$ consisting of Laplace eigenfunctions, there exists a zero-density exceptional set of basis eigenfunctions such  that
the remaining eigenfunctions satisfy QUE.  This property is known as {\bf Quantum Ergodicity}, and holds in great generality--- it depends only on the ergodicity of the flow.  QUE then asks that the exceptional set be empty, in the case of manifolds of negative sectional curvature.

For a number of reasons, it is suspected \cite{SarnakProgress} that spectral multiplicities (or lack thereof) are intimately tied with the QUE Conjecture, and it is this aspect that we wish to explore.
  We parametrize the spectrum of the Laplacian on $M$ by $\lambda = \frac{1}{4}+r^2$, and here $r\to\infty$ plays the role of the inverse Planck's constant in the semiclassical limit.  It is conjectured \cite{IwaniecSarnak} that the multiplicities are bounded by $O_\epsilon(\lambda^\epsilon)$, and in fact it would not be too surprising if they are uniformly bounded.  This problem is well out of the reach of current technology, but in any case it is not possible to study varying multiplicity scales and their effect on the quantum limits.  

However, one can artificially introduce degeneracies by considering quasimodes, or approximate eigenfunctions.  Define an $\omega(r)$-{\bf quasimode with approximate parameter} $r$ to be a function $\psi$ satisfying 
\begin{equation}\label{regular quasimode condition}
||(\Delta + (\frac{1}{4}+r^2))\psi||_2 \leq r \omega(r)||\psi||_2
\end{equation}
The  factor of $r$ in our definition comes from the fact that $r$ is essentially the square-root of the Laplace eigenvalue\footnote{Our definition here is different from the usual definition of quasimodes by a factor of $r$, since it is more convenient for us to consider the size of the spectral window $\omega(r)$ instead of the $L^2$ defect from being an eigenfunction.}.  For our purposes, one may think of an $\omega(r)$-quasimode like a linear combination of the eigenfunctions with spectral parameter in the window $[r-\omega(r), r+\omega(r)]$.  

The main term in Weyl's Law says that the asymptotic density of eigenfunctions near spectral parameter $r$ is proportional to $r$ (with a constant depending on the area of the surface).  Though controlling the error term is a very difficult problem--- indeed, this is precisely the problem of bounding multiplicities--- it is known that this approximation is valid for ``large logarithmic windows" $\omega(r)\geq C/\log{r}$, where again $C$ depends on the surface, and is believed to be valid for far smaller windows; eg., $\omega(r) = r^{-\delta}$.  In any case, we shall be interested in this paper primarily with windows of size $\omega(r) \gtrsim 1/\log{r}$, where we certainly expect the error in Weyl's Law to be small.  We conjecture that this is the appropriate threshold to study for QUE:
\begin{Conjecture}\label{hypconj}
Let $M$ be a compact hyperbolic surface.  Then any sequence of $o(1/\log{r})$-quasimodes satisfies QUE.
\end{Conjecture}

Within the context of hyperbolic surfaces, this conjecture is a slight strengthening of the QUE Conjecture \cite{RS},  essentially stating (in a quantitative way) that the small spectral multiplicities are responsible for QUE.  It is also important to observe that the window $\omega(r) = o(1/\log{r})$ is just beyond what can be analyzed at present:  as remarked by Sarnak \cite{SarnakProgress}, any proof of QUE is expected to address the multiplicity issue (though perhaps indirectly).  Conjecture~\ref{hypconj} suggests that QUE is on par with a multiplicity bound of $o(r/\log{r})$, whereas the current best known bound due to Berard \cite{Berard} is $O(r/\log{r})$ (with an explicit implied constant depending on the manifold).

We remark that, if true, Conjecture~\ref{hypconj} is sharp.  We have shown in \cite{locquasi} that for any compact hyperbolic surface $M$ and closed geodesic $\gamma\subset S^*M$, and any $\epsilon>0$, there exist $\frac{\epsilon}{\log{r}}$-quasimodes on $M$ whose microlocal lifts concentrate  positive mass on $\gamma$.  Naturally, the lower bound on the mass of $\gamma$ decreases to $0$ as $\epsilon\to 0$, though no attempt was made in the construction to obtain the sharpest possible quantitative bound.

\vspace{.2in}

In the present paper, we discuss quasimodes for Eisenstein series on $M=SL(2,\mathbb{Z})\backslash\mathbb{H}$ (though this may safely be replaced with similar congruence subgroups of $SL(2,\mathbb{Z})$).  We recall that the Eisenstein series are given by
$$E(z,s) = \sum_{\gamma\in\Gamma_\infty\backslash\Gamma} Im(\gamma.z)^{s}$$
for $Re(s)>1$, and analytically continued to be meromorphic in the entire plane.  We will be interested in the half-line $E(z,\frac{1}{2}+ir)$ as $r\to\infty$.
These are eigenfunctions are the Laplacian, of eigenvalue $\frac{1}{4} + r^2$, but are not in $L^2(M)$.  However, we still have following analogue of QUE, proven by Luo-Sarnak \cite{LuoSarnak} and Jakobson \cite{Jak}:  for fixed smooth functions $f, g \in C_0^\infty(S^*M)$ we have
\begin{equation}\label{LS-Jak}	
\frac{\mu_r(f)}{\mu_r(g)} = \frac{\int_{S^*M} f(x) dx}{\int_{S^*M} g(x) dx}	
\end{equation}
where 
$$\mu_r : f \mapsto \left\langle Op(f) E\left(z, \frac{1}{2}+ir\right), E\left(z,\frac{1}{2}+ir\right)\right\rangle$$
is a microlocal lift of $| E(z, \frac{1}{2}+ir)|^2 dz$ to $S^*M$, and $dx$ is Liouville measure on $S^*M$.  The renormalization is important, as they actually show that\footnote{The original paper \cite{LuoSarnak} has some typos, leading to an incorrect constant; we thank Yiannis Petridis for pointing  out this correction, as it appears in the paper \cite{PetridisRR, PetridisRR_Err}.}
$$\mu_r(f) \sim \frac{6}{\pi}\log{r} \cdot \int_{S^*M} f(x) dx$$
grows proportionally to $\log{r}$.

We wish to extend this result to the context of quasimodes, where we replace $E(z,\frac{1}{2}+ir)$ with a linear combination of Eisenstein series of varying spectral parameters, and construct the corresponding microlocal lifts.   Namely, we take a sequence $\{h_j\}$ of probability densities, i.e. $\int_r h_j(r)dr =1$, which we will assume to be smooth functions on $\mathbb{R}^+$, localized near $r_j$
 in the sense that
\begin{equation}\label{quasimode condition}
\int_r h_j(r)|r-r_j|dr = \omega(r_j) \searrow 0
\end{equation}
is small, and construct the ``Eisenstein $\omega(r_j)$-quasimodes"
$$E_{h_j}(z): = \int_r h_j(r) E(z, \frac{1}{2}+ir) dr$$

We then consider the microlocal lifts
$$\mu_{h_j} : f \mapsto \langle Op(f) E_{h_j}, E_{h_j}\rangle$$
and the weak-* limit points $\mu$ of the sequence $\{\frac{1}{2\log{r}}\mu_{h_j}\}$ as $r_j\to\infty$, which we call quantum limits for $\{E_{h_j}\}$.  We have chosen the normalization such that true Eisenstein series give probability measures; by \cite{LuoSarnak} and \cite{Jak} we have
$$\mu_{r}(S^*M) = \frac{6}{\pi}  \log{r} \cdot Vol(S^*M) = 2\log{r}$$
since the volume of the modular surface is $\frac{\pi}{3}$.

Clearly, studying these quantum limits boils down to calculating off-diagonal matrix coefficients of the form 
$$\mu_{r',r''}(f) := \left\langle Op(f)E(z, \frac{1}{2}+ir'), E(z, \frac{1}{2}+ir'')\right\rangle$$
and then integrating back over $r'$ and $r''$.  For the most part we will think of $h_j$ as being non-negative, though this is not essential for the results.  In fact, it is useful to think of the $h_j$ as being scalings of a fixed non-negative, symmetric bump function $\chi$; i.e., set $h_j(r) = \chi\left( \frac{r-r_j}{K_j}\right)$ for a suitable choice of scaling, such as $K_j=C/\log{r_j}$.

Our main results in this context are:
\begin{Theorem}\label{equidist}
Let $\mu$ be a quantum limit for Eisenstein $o(1)$-quasimodes; i.e., a weak-* limit point of the distributions $\{\frac{1}{2\log{r}}\mu_{h_j}\}$ for a sequence $\{h_j\}$ satisfying $\int_r h_j(r)dr = 1$ and $\int_r h_j(r)|r-r_j|dr = o(1)$ as $r_j\to\infty$.  

Then for any  $f, g\in C_0^\infty(S^*M)$ we have
$$\mu(f) \int_{S^*M} g(x) dx = \mu(g)\int_{S^*M} f(x) dx   $$
\end{Theorem}
This is the appropriate statement of QUE for our context; we have adjusted the statement slightly from (\ref{LS-Jak}) to allow for $\mu$ to be the zero measure, which can happen in the context of quasimodes.

For Eisenstein $O(1/\log{r})$-quasimodes, we can calculate the mass of $\mu$ and say more:
\begin{Theorem}[Maass-Selberg relations] \label{sharp asymptotic}
Let $r_j\to\infty$, and suppose we have a function $\omega(r_j)=o(\log\log{r_j}/\log{r_j})$.  Then for any $r', r''$ satisfying
$$|r'-r_j|, |r''-r_j| <\omega(r_j)$$ 
we have
$$\frac{\mu_{r', r''}(f)}{2\log{r_j}} = \frac{e^{2i\log{r_j}\cdot (r'-r'')}-1}{2i\log{r_j}\cdot (r'-r'')} \cdot \frac{3}{\pi}\int_{S^*M}f(x) dx + o(1)$$

Thus if $\{h_j\}$ satisfies $\int_r h_j(r)dr=1$ and $\int_r h_j(r)|r-r_j|dr =o(\log\log{r_j}/\log{r_j})$, then we have
$$\frac{\mu_{h_j}(f)}{2\log{r_j}} = \frac{3}{\pi} \int_{S^*M}f(x)dx 
\cdot \iint h_j(r')\overline{h_j}(r'') 
\frac{e^{2i\log{r_j}\cdot (r'-r'')}-1}{2i\log{r_j}\cdot (r'-r'')}
dr'dr'' + o(1)$$
and if in fact $\int_rh(r)|r-r_j|dr =o(1/\log{r_j})$, then we get
\begin{eqnarray*}
\frac{\mu_{h_j}(f)}{2\log{r_j}} & \sim & \frac{3}{\pi}  \int_{S^*M}f(x) dx
\end{eqnarray*}
\end{Theorem}
In other words, as far as quantum limits of microlocal lifts are concerned, those of Eisenstein $o(1/\log{r})$-quasimodes are indistinguishable, in the limit $r\to\infty$, from those of true Eisenstein series.  But if we weaken the condition and consider $\left(\frac{C}{\log{r}}\right)$-quasimodes for a constant $C>0$, then the total mass can be diminished, with extra localization in the cusp.  This last aspect will be explored in greater detail in section~\ref{threshold}.

Another consequence of Theorem~\ref{sharp asymptotic} is the following
\begin{Corollary}[\cite{meRoman}]\label{Ehrenfest cutoff}
Let $\{h_j\}$ satisfy the quasimode condition (\ref{quasimode condition}) with $\omega(r_j)=o(\log\log{r_j}/\log{r_j})$, and $\int_r h_j(r)dr=1$.  Then
the microlocal lift $\mu_{h_j}$ is asymptotic to
$$\mu_{h_j}(f)  \sim \frac{3}{\pi}\int_{S^*M} f(x)dx  \cdot  \int_{t=0}^{ 2\log{r_j}}|\widehat{h_j}(t)|^2 dt$$
\end{Corollary}

Since the volume of the modular surface is $\frac{\pi}{3}$, the quantity $\frac{3}{\pi}\int f(x)dx$ is the normalized average of $f$ on $S^*M$; thus, in other words, Corollary~\ref{Ehrenfest cutoff} says that the total mass of $\mu_{h_j}$ that does not escape to the cusp is given by $\int_{t=0}^{2\log{r_j}}|\widehat{h_j}(t)|^2 dt$.

Since $\widehat{h_j}$ represents wave propagation times, we interpret this result as saying that the Eisenstein plane waves follow geodesic paths and equidistribute on the surface up to twice the Ehrenfest time $2T_E=2\log{r_j}$, and then ``immediately afterwards" disengage from their geodesic paths and retreat back to the cusp.  That the plane waves should escape to the cusp so quickly is somewhat peculiar, and is the topic of further work with Roman Schubert \cite{meRoman}.

\vspace{.2in}

Though it may appear at first glance that Theorem~\ref{equidist} and Theorem~\ref{sharp asymptotic} disagree regarding the threshold for QUE, the resolution lies in the proportionality constants implicit in (\ref{LS-Jak}).  Up to windows of size $\omega(r) = o(1/\log{r})$, we not only have the equidistribution implied by Theorem~\ref{equidist}, but the implicit proportionality constant is preserved, according to Theorem~\ref{sharp asymptotic}.  Once this threshold is passed and we move into larger logarithmic windows $\omega(r) \gtrsim 1/\log{r}$, this proportionality constant drops--- due to the introduction of interferences that decrease the mass of the measures in compact sets, with more mass localizing in the cusp.  However, Theorem~\ref{equidist} says that the  remaining mass still equidistributes in compact sets all the way out  to massive windows of size $\omega(r) = o(1)$.  This is as large a window as one can take and still hope to have QUE, as it is straightforward to construct examples in windows of size $\gtrsim 1$ that do not equidistribute (see also \cite{localized_example} for a stronger statement regarding arbitrary subspaces of dimension $\gtrsim r$).

The $o(1)$-quasimodes allowed by Theorem~\ref{equidist} are significantly weaker than the logarithmic windows of Conjecture~\ref{hypconj}--- and the threshold at which QUE is known to fail on compact hyperbolic surfaces \cite{locquasi}.  The reason for this is explored in section~\ref{threshold}, and can be understood in terms of the inherent arithmetic structure of the Eisenstein series on congruence surfaces.  There is also a dynamical-geometric aspect, not unrelated to the arithmetic, of the cusp at infinity.  In fact, we will see that there is indeed an analog of the result of \cite{locquasi} at the $o(1/\log{r})$ threshold, except that the localization is on cusp-bound geodesics, whose mass escapes to the cusp and thus does not appear in the asymptotics that yield Theorem~\ref{equidist}.  Lastly, there is another special spectral property of congruence surfaces that is intimately tied to the arithmetic, and highly relevant to these results; namely, the abundance of cusp forms, and sparsity of the Eisenstein series, first observed by Selberg.  We would expect that on a generic non-compact surface of finite volume, where the continuous spectrum is expected to be dominant (see \cite{PhillipsSarnak}), Theorem~\ref{equidist} should not hold.

\vspace{.2in}

The proofs of Theorems~\ref{equidist} and \ref{sharp asymptotic} are based on the papers of Luo-Sarnak \cite{LuoSarnak} and Jakobson \cite{Jak} on (true) Eisenstein series, in conjunction with the following  observation, which is an immediate consequence of Stirling's formula for the Gamma function:
\begin{Lemma}\label{Gamma}
Let $r_j\to\infty$, and suppose we have a function $\omega(r)=o(1)$.  Then for any $r', r''$ satisfying
$$|r'-r_j|, |r''-r_j| <\omega(r_j)$$ 
we have,
for any fixed $\sigma\in\mathbb{R}$,
$$\frac{|\Gamma(\sigma+ir')|} { |\Gamma(\sigma+ir'')|} = 1 + o_\omega(1)$$
with an error term $o_\omega(1)\to0$ as $r_j\to\infty$, depending on the decay rate of $\omega$.  

Moreover, if $\omega(r)=o(\log\log{r}/\log{r})$, then we have
$$\frac{\Gamma(\sigma +ir')}{ \Gamma(\sigma+ir'')} = e^{i(r'-r'')\log{r_j}} + o_\omega(1)$$
and in particular, if in fact $\omega(r) = o(1/\log{r})$, then
$$\frac{\Gamma(\sigma+ir')}{\Gamma(\sigma+ir'')} = 1+o_\omega(1)$$
\end{Lemma}

In a nutshell, Lemma~\ref{Gamma} allows one to push through all of the estimates of \cite{LuoSarnak} and \cite{Jak} with $\mu_{r', r''}$ in place of $\mu_{r_j}$, as it becomes apparent that the 
greatest sensitivity to changes in $r', r''$ comes from the Gamma factors involved.  Thus all estimates of decaying error terms (where all factors involved are estimated in absolute value) hold whenever $r', r''=r_j+o(1)$.  The most delicate estimates occur when computing  $\mu_{r', r''}$ against incomplete Eisenstein series; but even without taking absolute values, all terms retain their asymptotics  for $r',r'' = r_j+ o(\log\log{r_j}/\log{r_j})$, save for some Gamma factors whose phase changes in accordance with Lemma~\ref{Gamma}.  Thus we get the precise asymptotics of Theorem~\ref{sharp asymptotic}.

\section{Fourier Expansions}\label{expansions}

We will need the Fourier expansions of our various  test functions:  holomorphic cusp forms of weight $2k$, ``shifted" Maass forms of eigenvalue $\frac{1}{4}+t_j^2$ and weight $2k$, and Eisenstein series of weight $2k$.  We will test the equidistribution of $\mu_{r',r''}$ by integrating it against holomorphic cusp forms of weight $2k\neq 0$ and Maass cusp forms of eigenvalue $\neq 0$, which are orthogonal to constants, and incomplete Eisenstein series of weight $2k$.  Together these span a dense subspace of $L^2(S^*M)$, and Theorem~\ref{equidist} will follow by showing that the integral of the first two against $d\mu_{r',r''}$ decay, while the integral against an incomplete Eisenstein series is asymptotically proportional to its average.

In this section we recall the relevant Fourier expansions that we will need; for further details, see \cite{Jak}.

\subsection{Holomorphic cusp forms}
A holomorphic cusp form of weight $2k$ is a holomorphic function satisfying
$$F(\gamma z) = (cz+d)^{2k}F(z)$$
for all $\gamma = \begin{pmatrix}	\ast & \ast\\ c & d	\end{pmatrix} \in PSL(2,\mathbb{Z})$.
$F$ is assumed to be holomorphic and vanishing at infinity, and so has a Fourier expansion
$$F(z) = \sum_{n>0} c(n) e(nz) = \sum_{n>0} c(n) e^{-2\pi ny}e(nx)$$
where we have used the notation $e(w) = e^{2\pi i w}$.  Then the function
$$f(x,y,\theta) = e^{-2ik\theta}y^k F(x,y) = e^{-2ik\theta}y^k \sum_{n>0} c(n) e^{-2\pi ny} e(nx)$$
is a Casimir eigenfunction of eigenvalue $k(1-k)$ on $S^*M$.  We may (and will) also assume that $F$ is an eigenfunction of all Hecke operators, whereby the coefficients $c(n)$ are multiplicative, and we may form the $L$-function
$$L(F, s) := \sum_{n=1}^\infty \frac{c(n)}{n^s}$$
which has analytic continuation and functional equation, with critical line $Re(s)=k$.

Note that $k=0$ forces $f$ to be constant, and so we need only concern ourselves with holomorphic forms of weight $k\neq 0$.

\subsection{Shifted Maass cusp forms}

Let $\varphi_j$ be a Hecke-Maass cusp form; that is, an eigenfunction of Laplacian in $PSL(2,\mathbb{Z})\backslash\mathbb{H}$, of Laplace eigenvalue $\frac{1}{4}+t_j^2$, that is also an eigenfunction of all Hecke operators.  Writing its Fourier expansion and applying separation of variables, we obtain an ordinary differential equation for the coefficients, whose solution yields the expansion
$$\varphi_j(x,y) = \sum_{n\neq 0} \frac{c_j(|n|)}{\sqrt{|n|}} W_{0, it_j}(4\pi |n|y) e(nx)$$
where $W_{k,\mu}$ are Whittaker functions (for $k=0$ these are scaled Bessel functions).  Applying raising and lowering operators gives additional Casimir eigenfunctions on $L^2(S^*M)$, whose Fourier expansions are then given by \cite[(1.8)]{Jak}
\begin{eqnarray*}
\varphi_{j,k} & = & \frac{(-1)^k\Gamma(\frac{1}{2}+it_j)}{\Gamma(\frac{1}{2}+ k+it_j)}\sum_{n>0}\frac{c_j(|n|)}{\sqrt{|n|}} W_{k,it_j}(4\pi |n|y) e(nx)\\
& & \pm \frac{(-1)^k\Gamma(\frac{1}{2}+it_j)}{\Gamma(\frac{1}{2}- k+it_j)}\sum_{n<0}\frac{c_j(|n|)}{\sqrt{|n|}} W_{-k,it_j}(4\pi |n|y) e(nx)
\end{eqnarray*}
where the sign in the second line is positive if $\varphi$ is an even Maass form, and negative if $\varphi$ is odd.

Since the coefficients $c_j(n)$ are again Hecke eigenvalues and therefore multiplicative, we form the $L$-function
$$L(\varphi_j, s) = \sum_{n=1}^\infty \frac{c_j(n)}{n^s}$$
which again has analytic continuation and functional equation, with critical line $Re(s)=\frac{1}{2}$.

\subsection{Eisenstein series}

For weight $0$ Eisenstein series, we have the expansion \cite[section 3.4]{Iwaniec}
\begin{equation}\label{eisen 0 expansion}
E(z,\frac{1}{2}+ir) =  y^{\frac{1}{2}+ir} + \phi(\frac{1}{2}+ir)y^{\frac{1}{2}-ir} + \frac{4\sqrt{y}}{\xi(1+2ir)}\sum_{n=1}^\infty n^{ir}\sigma_{-2ir}(n) K_{ir}(2\pi ny)\cos(2\pi nx)
\end{equation}
writing, here and throughout, that $z=x+iy$ on the upper-half plane.

Finally, if $k> 0$, we have for Eisenstein series of weight $\pm 2k$ \cite[(1.6) and (1.7)]{Jak} 
\begin{eqnarray}
E_{-2k} (z, 1/2+ir) & = & y^{1/2+ir} + \frac{(-1)^k\Gamma^2(\frac{1}{2}+ir)}{\Gamma(\frac{1}{2}-k+ir)\Gamma(\frac{1}{2}+k+ir)}\phi(\frac{1}{2}+ir)y^{1/2-ir}\nonumber\\
& & + \frac{(-1)^k\Gamma(\frac{1}{2}+ir)}{2\Gamma(\frac{1}{2}+k+ir)\xi(1+2ir)} \sum_{n>0} \frac{|n|^{ir}\sigma_{-2ir}(|n|)}{\sqrt{|n|}}W_{k, ir}(4\pi|n|y)e(nx)\nonumber\\
& & + \frac{(-1)^k\Gamma(\frac{1}{2}+ir)}{2\Gamma(\frac{1}{2}-k+ir)\xi(1+2ir)} \sum_{n<0} \frac{|n|^{ir}\sigma_{-2ir}(|n|)}{\sqrt{|n|}}W_{-k, ir}(4\pi|n|y)e(nx)\label{eisen -2k expansion}\nonumber\\
\\
E_{2k} (z, 1/2+ir) & = & y^{1/2+ir} + \frac{(-1)^k\Gamma^2(\frac{1}{2}+ir)}{\Gamma(\frac{1}{2}-k+ir)\Gamma(\frac{1}{2}+k+ir)}\phi(\frac{1}{2}+ir)y^{1/2-ir}\nonumber\\
& & + \frac{(-1)^k\Gamma(\frac{1}{2}+ir)}{2\Gamma(\frac{1}{2}+k+ir)\xi(1+2ir)} \sum_{n<0} \frac{|n|^{ir}\sigma_{-2ir}(|n|)}{\sqrt{|n|}}W_{k, ir}(4\pi|n|y)e(nx)\nonumber\\
& & + \frac{(-1)^k\Gamma(\frac{1}{2}+ir)}{2\Gamma(\frac{1}{2}-k+ir)\xi(1+2ir)} \sum_{n>0} \frac{|n|^{ir}\sigma_{-2ir}(|n|)}{\sqrt{|n|}}W_{-k, ir}(4\pi|n|y)e(nx)\nonumber\\
\label{eisen 2k expansion}
\end{eqnarray}
where here and throughout,
\begin{eqnarray*}
\xi(s) & = & \pi^{-s/2}\Gamma(\frac{s}{2})\zeta(s)\\
\phi(s) & = & \frac{\xi(2s-1)}{\xi(2s)}\\
\sigma_s(n) & = & \sum_{d|n}d^s
\end{eqnarray*}

With these definitions, we can write the microlocal lift $\mu_{r', r''}$  as the distribution
$$\int_{S^*M} f d\mu_{r',r''} := 
\int_{S^*M} f(z,\theta) E(z,\frac{1}{2}-ir')\sum_{k=-\infty}^\infty E_{2k}(z,\frac{1}{2}+ir'')e^{2ik\theta} \frac{dxdy}{y^2} d\theta$$

\section{Integral Against Shifted Maass Forms and Holomorphic Cusp Forms}

For $r'=r''=r$, the integral of $d\mu_r$ against holomorphic cusp forms and shifted Maass cusp forms was shown in \cite{LuoSarnak} and \cite{Jak} to decay like $|r|^{-1/6+\epsilon}$, due to subconvex estimates for the relevant $L$-functions.  This polynomial decay in $|r|$ is certainly robust enough to withstand our small variations in $r'$ and $r''$, and the arguments of \cite{LuoSarnak} and \cite{Jak} go through without change.  We bring here the argument for holomorphic cusp forms, and the proof for shifted Maass cusp forms follows analogously, using the estimates of \cite{Jak} for the hypergeometric series involved.  The main thrust of the argument is the factoring of a degree 4 $L$-function into a product of two degree 2 $L$-functions--- corresponding to the form $F$--- for which we have a non-trivial subconvexity bound, due to Good \cite{Good}.  The parallel subconvexity bound for Maass forms is due to Meurman \cite{Meurman}.

\begin{Proposition}{\em (see \cite[Proposition 2.1]{Jak})}
Let $r_j\to\infty$, and suppose we have a function $\omega(r_j)=o(1)$.  Then for any $r', r''$ satisfying
$$|r'-r_j|, |r''-r_j| <\omega(r_j)$$ 
we have,
for any weight $2k$ holomorphic  cusp form $F$, that
$$\int_{S^*M} F d\mu_{r', r''} \lesssim_{F, k,\epsilon,\omega} |r_j|^{-1/6+\epsilon}$$
\end{Proposition}

\begin{proof}
We have for $Re(s)>1$ 
\begin{eqnarray*}
\lefteqn{\int_{\Gamma\backslash\mathbb{H}}F(z,0) E(z, \frac{1}{2}-ir') E_{-2k}(z, s) dArea}\\
& = & \int_{\Gamma\backslash\mathbb{H}} F(z,0) E(z, \frac{1}{2}-ir')  \sum_{\gamma\in\Gamma_\infty\backslash\Gamma} Im(\gamma z)^{s}(\epsilon_\gamma(z))^{2k} dArea\\
& = & \sum_{\gamma\in\Gamma_\infty\backslash\Gamma} \int_{\gamma(\Gamma\backslash\mathbb{H})} F(x,y) E((x,y), \frac{1}{2}-ir') y^{s} \frac{dxdy}{y^2}\\
& = & \int_{y=0}^\infty \int_{x=0}^1 y^{s} F(x,y) E((x,y),\frac{1}{2}-ir') \frac{dxdy}{y^2}
\end{eqnarray*}
Via analytic continuation, we may substitute $s=\frac{1}{2}+ir''$, and use the orthogonality of the weight spaces to obtain (after integrating out $\theta$)
\begin{eqnarray*}
\int_{S^*M} F  d\mu_{r',r''} & = & \int_{\Gamma\backslash\mathbb{H}}F(z,0) E(z, \frac{1}{2}-ir') E_{-2k}(z, \frac{1}{2}+ir'') dArea\\
& = & \int_{y=0}^\infty \int_{x=0}^1 y^{\frac{1}{2}+ir''} F(x,y) E((x,y),\frac{1}{2}-ir') \frac{dxdy}{y^2}
\end{eqnarray*}

We now substitute the Fourier expansions for $F(\cdot ,0)$ and $E(\cdot, \frac{1}{2}-ir')$ and integrate out $x$ (notice this eliminates the $x$-independent terms from the expansion of the Eisenstein series), giving
\begin{eqnarray*}
\lefteqn{\int_{S^*M} F  d\mu_{r',r''}}\\
& = & \frac{1}{2\xi(1-2ir')}\sum_{n=1}^\infty \int_0^\infty y^{k+\frac{1}{2}+ir''}\frac{c(n)n^{-ir'}\sigma_{2ir'}(n)}{\sqrt{n}}e^{-2\pi ny} W_{0,-ir'}(4\pi ny)\frac{dy}{y^2}\\
& = & \frac{(4\pi)^{-k+\frac{1}{2}-ir''}}{2\xi(1-2ir')} 
\left(	\sum_{n=1}^\infty\frac{c(n)n^{-ir'}\sigma_{2ir'}(n)}{n^{k+ir''}}\right)
\left( \int_0^\infty W_{0, -ir'}(u) e^{-u/2}u^{k-\frac{1}{2}+ir''}\frac{du}{u}	\right)
\end{eqnarray*}
by changing variables to $u=4\pi ny$.  

The integral can be evaluated, and is equal to \cite[7.621.11]{G-R}
$$\int_0^\infty W_{0, -ir'}(u) e^{-u/2}u^{(k-\frac{1}{2}+ir'')-1} du = \frac{\Gamma(k+i\Sigma r)\Gamma(k-i\Delta r)}{\Gamma(k+\frac{1}{2}+ir'')}$$
where we have introduced the shorthand 
\begin{eqnarray*}
\Sigma r & = & r'+r''\\
\Delta r & = & r'-r'' 
\end{eqnarray*}
to be used throughout the paper.

The infinite sum can be factored as in \cite{LuoSarnak} to be
$$\sum_{n=1}^\infty\frac{c(n)n^{-ir'}\sigma_{2ir'}(n)}{n^{k+ir''}} 
= \frac{L(F,k+i\Sigma r)L(F, k-i\Delta r)}{\zeta(2k+2ir'')}$$

Therefore putting everything together we get
\begin{eqnarray*}
\int_{S^*M} F d\mu_{r',r''} & = & \frac{(4\pi)^{-k+\frac{1}{2}-ir''}}{2\pi^{\frac{1}{2}-ir'}} \frac{L(F, k+i\Sigma r)L(F, k-i\Delta r)}{\zeta(1-2ir')\zeta(2k+2ir'')}\frac{\Gamma(k+i\Sigma r)\Gamma(k-i\Delta r)}{\Gamma(k+\frac{1}{2}+ir'')\Gamma(\frac{1}{2}-ir')}\\
& \lesssim & \frac{L(F, k+i\Sigma r)}{\zeta(1-2ir')}\frac{\Gamma(k+i\Sigma r)}{\Gamma(k+\frac{1}{2}+ir'')\Gamma(\frac{1}{2}-ir')}
\end{eqnarray*}
since $k\geq 1$ means that $\zeta(2k+2ir'')$ is bounded uniformly as $r_j\to\infty$, as is $\Gamma(k-i\Delta r)$ and $L(F,k-i\Delta r)$.  Estimating the remaining gamma factors by Stirling's approximation
$$|\Gamma(\sigma+ir)| \sim \sqrt{2\pi}e^{-\pi |r|/2}|r|^{\sigma-\frac{1}{2}}$$
we get
$$\int_{S^*M} F d\mu_{r',r''} \lesssim \frac{L(F, k+i\Sigma r)}{\zeta(1-2ir')} |r_j|^{-1/2}$$
By the estimate \cite{Titchmarsh}
$$|\zeta(1+ir)| \gtrsim 1/\log{r}$$
and the subconvexity bound \cite{Good}
$$|L(F,k+ir)| \lesssim_{F, \epsilon} |r|^{1/3+\epsilon}$$
due to Good, the Proposition follows.
\end{proof}

\section{Integral Against Incomplete Eisenstein Series}

\subsection{Weight $0$ Incomplete Eisenstein Series --- The Main Term}\label{luo sarnak}

As in \cite{LuoSarnak} and \cite{Jak}, the most delicate analysis is evaluating $\int_{S^*M} F_\psi d\mu_{r',r''} $  where
\begin{eqnarray}
F_\psi(z,\theta) & = & \frac{1}{2\pi i}  \int_{Re(s)=2} L_\psi(s) E(z,s) ds\label{incomplete eisen 0}\\
L_\psi(s) & = &\int_{y=0}^\infty \psi(y) y^{-s}\frac{dy}{y}\nonumber
\end{eqnarray}
is an incomplete Eisenstein series (of weight $0$).  As before, we consider 
$$\left\langle Op(F_\psi) E(z,\frac{1}{2}-ir'), E(z, \frac{1}{2}+ir'')\right\rangle = \int_{\Gamma\backslash\mathbb{H}} F_\psi(z,0)E(z, \frac{1}{2}-ir')E(z, \frac{1}{2}+ir'')\frac{dxdy}{y^2}$$
noting that only the weight $0$ Eisenstein series appear, due to the orthogonality of the weight spaces and the fact that $F$ is weight $0$.  
We then substitute (\ref{incomplete eisen 0}) for $F_\psi$, and after ``unfolding" the $E(z,s)$'s in the expansion (\ref{incomplete eisen 0}) we get
\begin{eqnarray*}
\int F_\psi d\mu_{r', r''} & = & \frac{1}{2\pi i} \int_{Re(s)=2} \int_{y=0}^\infty \int_{x=0}^1 L_\psi(s) y^s E(z, 1/2-ir') E(z,1/2+ir'') \frac{dxdy}{y^2}
\end{eqnarray*}

We replace $E(z,\frac{1}{2}-ir')$ and $E(z,\frac{1}{2}+ir'')$ with their Fourier expansions (\ref{eisen 0 expansion}), and after integrating out $x$ we get
\begin{eqnarray*}
\lefteqn{\int F_\psi d\mu_{r', r''} }\\
& = & \frac{1}{2\pi i} \int_{Re(s)=2} \int_0^\infty y^sL_\psi(s) \frac{dy}{y}ds\cdot \Big(1+\phi(1/2-ir')\phi(1/2+ir'')\Big)\\
& & + \Big(\text{rapidly decreasing in } r_j\Big)\\
& & + \frac{4}{\pi i \xi(1-2ir')\xi(1+2ir'')}\int_{Re(s)=2}L_\psi(s) \left(\sum_{n=1}^\infty \frac{\sigma_{2ir'}(n)\sigma_{-2ir''}(n)}{n^{i\Delta r}}\right)\times\\
& &\quad \quad \times \int_0^\infty K_{-ir'}(2\pi ny)K_{ir''}(2\pi ny)y^s \frac{dy}{y}
\end{eqnarray*}
where the ``rapidly decreasing in $r_j$" term comes from integrating the highly oscillatory cross-terms
$$\frac{1}{2\pi i}\int_{Re(s)=2} \int_0^\infty y^{s-1\pm i(r'+r'')}L_\psi(s) dy ds$$

Now by the Mellin inversion formula 
$$\psi(y) = \frac{1}{2\pi i } \int_{Re(s)=2} L_\psi(s) y^s ds$$
and the fact that for any $r\in\mathbb{R}$ we have
\begin{eqnarray*}
\left| \phi(\frac{1}{2}+ir)\right|  & = & \left| \frac{\xi(2ir)}{\xi(1+2ir)}\right| = \left| \frac{\xi(1-2ir)}{\xi(1+2ir)}\right| =1 
\end{eqnarray*}
by the functional equation $\xi(s) = \xi(1-s)$ and $\xi(\bar{s}) = \overline{\xi(s)}$, we see that the first term
$$\left|	 \int_{Re(s)=2}\int_{0}^\infty y^sL_\psi(s) \frac{dy}{y}ds\cdot \Big( 1 + \phi(\frac{1}{2}-ir')\phi(\frac{1}{2}+ir'')\Big)	\right| \leq 2\int_0^\infty \psi(y) \frac{dy}{y} $$
is uniformly bounded in $r_j$; this is the contribution from the incoming and outgoing components of the Eisenstein series (which concentrate in the cusp).

Therefore, after interchanging summation and integration and changing variables to $u= ny$, we get
\begin{eqnarray*}
\lefteqn{\int F_\psi d\mu_{r',r''} + O(1)}\\
& = & \frac{4}{\pi i \xi(1-2ir')\xi(1+2ir'')}\int_{Re(s)=2}L_\psi(s) \left(\sum_{n=1}^\infty \frac{\sigma_{2ir'}(n)\sigma_{-2ir''}(n)}{n^{s+i\Delta r}}\right)\times\\
& &\quad \quad \times \left(\int_0^\infty K_{-ir'}(u)K_{ir''}(u)u^s \frac{du}{u}\right)
\end{eqnarray*}

The infinite sum can be evaluated, as was first done by Ramanujan \cite{Ramanujan_sum_formulas}
$$\sum_{n=1}^\infty \frac{\sigma_{2ir'}(n)\sigma_{-2ir''}(n)}{n^{s+i\Delta r}} = \frac{\zeta(s+i\Delta r) \zeta(s-i\Sigma r)\zeta(s+i\Sigma r)\zeta(s-i\Delta r)}{\zeta(2s)}$$
The integral can also be evaluated \cite[6.576.4]{G-R}  to be
$$\int_0^\infty K_{-ir'}(2\pi u)K_{ir''}(2\pi u) u^{s-1}du 
=  \frac{\Gamma(\frac{s-ir'+ir''}{2})\Gamma(\frac{s-ir'-ir''}{2})
\Gamma(\frac{s+ir'+ir''}{2})\Gamma(\frac{s+ir'-ir''}{2})}{8\pi^s\Gamma(s)}$$

We shift the line of integration to $Re(s)=\frac{1}{2}$, using Stirling's approximation
$$|\Gamma(\sigma + ir)| \sim \sqrt{2\pi}e^{-\pi|r|/2}|r|^{\sigma-\frac{1}{2}}$$
to estimate the $\Gamma$ factors and Weyl's subconvexity estimate 
$$|\zeta(\frac{1}{2} +ir)|\lesssim_\epsilon |r|^{1/6+\epsilon}$$
together with the bound \cite{Titchmarsh}
$$|\zeta(1+ir)|\gtrsim 1/ \log{r}$$
to show that the integral over $Re(s)=1/2$ is bounded by $O(|r_j|^{-1/6+\epsilon})$.  The residues at the poles $s=1\pm i\Sigma r$  also decay due to the fact that $L_\psi$ is rapidly decreasing.  Thus it remains to compute the asymptotics of the residues at $s=1\pm i\Delta r$.

We write the integrand as 
$$\zeta(s+i\Delta r)\zeta(s-i\Delta r) B_0(s)$$
where $B_0(s)$ is holomorphic near $s=1$, which leads to the expression
$$B_0(s) = \frac{1}{2}L_\psi(s)\frac{\zeta(s-i\Sigma r)\zeta(s+i\Sigma r)
\Gamma(\frac{s-i\Delta r}{2})\Gamma(\frac{s-i\Sigma r}{2})\Gamma(\frac{s+i\Sigma r}{2})\Gamma(\frac{s+i\Delta r}{2})}
{i\pi^{s+i\Delta r}\Gamma(\frac{1}{2}-ir')\Gamma(\frac{1}{2}+ir'')\zeta(1-2ir')\zeta(1+2ir'')\zeta(2s)\Gamma(s)}
$$

\begin{Lemma}
We have
$$B_0(1-i\Delta r) \sim \frac{3}{i\pi^2}\int_{S^*M}  F_\psi \frac{dxdy}{y^2}	$$
\end{Lemma}

\begin{proof}
Substituting $s=1-i\Delta r$ causes the two $\zeta$ factors on top to cancel with the first two on the bottom, and similarly the first two $\Gamma$ factors on the bottom cancel with the $\Gamma(\frac{s\pm i\Sigma r}{2})$ terms on top.  We are left with
\begin{eqnarray*}
B_0(1-i\Delta r) & = & \frac{1}{2}L_\psi(1-i\Delta r)	\frac{\Gamma(\frac{1}{2}-i\Delta r)\Gamma(\frac{1}{2})}
{i\pi \zeta(2-2i\Delta r)\Gamma(1-i\Delta r)}\\
& \sim & 3L_\psi(1)\frac{\pi}{i\pi \cdot \pi^2 }
\end{eqnarray*}
by letting $\Delta r\to 0$  (all remaining $\zeta$ and $\Gamma$ factors have bounded arguments) and substituting the values $\Gamma(\frac{1}{2}) = \sqrt{\pi}$ and $\zeta(2)=\pi^2/6$.  The Lemma now follows from the fact that
$$L_\psi(1) = \int_0^\infty \psi(y) \frac{dy}{y^2} = \int_{S^*M} F_\psi(x,y) \frac{dxdy}{y^2}$$
\end{proof}

\begin{Lemma}\label{|B_0|}
We have
$$B_0(1+i\Delta r) \sim \frac{\xi(1+2ir')\xi(1-2ir'')}{\xi(1+2ir'')\xi(1-2ir')}B_0(1-i\Delta r)$$
In particular, since $\left|\frac{\xi(s)}{\xi(\bar{s})}\right| = 1$, we have $|B_0(1+i\Delta r)| \sim |B_0(1-i\Delta r)|$.

Moreover, if $\Delta r  =o(\log\log{r_j} /\log{r_j})$, then 
$$B_0(1+i\Delta r) \sim e^{2i\Delta r\log{r_j}} B_0(1-i\Delta r)$$
\end{Lemma}

\begin{proof}
The first statement follows from substituting $s=1+i\Delta r$ and following the previous Lemma:  we take $\Delta r\to 0$ in all of the bounded arguments, and we are simply left with $r'$ and $r''$ interchanged in the unbounded $\Gamma$ and $\zeta$ factors.  The second statement follows from $|\xi(\bar{s})| = |\overline{\xi(s)}|=|\xi(s)|$.

For the last statement, we need to evaluate asymptotics of the phase prefactor 
$$e^{i\theta} = \frac{\zeta(1+2ir')\zeta(1-2ir'')\Gamma(\frac{1}{2}+ir')\Gamma(\frac{1}{2}-ir'')}{\zeta(1+2ir'')\zeta(1-2ir')\Gamma(\frac{1}{2}+ir'')\Gamma(\frac{1}{2}-ir')}$$
so that 
$$B_0(1+i\Delta r) \sim e^{i\theta}B_0(1-i\Delta r)$$
Taking logarithms, we see
\begin{eqnarray*}
i\theta & = & \Big( \log{\zeta(1+2ir')} - \log{\zeta(1+2ir'')}\Big) 
+ \Big(\log{\Gamma(\frac{1}{2}+ir')} - \log{\Gamma(\frac{1}{2}+ir'')}		\Big)\\
& & \quad + \Big( \log{\zeta(1-2ir'')} - \log{\zeta(1-2ir')}\Big) 
+ \Big(\log{\Gamma(\frac{1}{2}-ir'')} - \log{\Gamma(\frac{1}{2}-ir')}		\Big)
\end{eqnarray*}
Applying the Weyl-Hadamard-de la Vall\'ee Poussin bound \cite{Titchmarsh}
$$\frac{d}{dr}\log{\zeta(1+ir)} \lesssim \frac{\log{r}}{\log\log{r}}$$
and Stirling's formula
$$\frac{d}{dr} \log{\Gamma(\frac{1}{2}+ir)} = \log{r} + O(1)$$
we find 
$$i\theta = 2(O(\log{r_j}/\log\log{r_j})+ \log{r_j}) \cdot i\Delta r = 2i\Delta r\log{r_j} + o(1)$$
whenever $\Delta r = o(\log\log{r_j}/\log{r_j})$, and the final result follows.
\end{proof}

\vspace{.1in}

We can now complete the calculation of the residues at $s= 1 \pm i\Delta r$.  Writing the Laurent expansion
$$\zeta(s) = \frac{1}{s-1} + \gamma + \sum_{n=1}^\infty c_n (s-1)^n$$
of $\zeta$ near $s=1$, we have
\begin{eqnarray*}
\lefteqn{Res_{s=1+i\Delta r} + Res_{s=1-i\Delta r}}\\ 
& = & \zeta(1+2i\Delta r)B_0(1+i\Delta r) + \zeta(1-2i\Delta r)B_0(1-i\Delta r)\\
& = & \left(	\frac{1}{2i\Delta r} +O(\Delta r)	\right) \Big(B_0(1+i\Delta r)-B_0(1-i\Delta r)\Big)\\
& &  \quad\quad + \left(	\gamma +O(\Delta r)^2	\right) \Big(B_0(1+i\Delta r)+B_0(1-i\Delta r)\Big)\\
 & \sim & \left(\frac{e^{i\theta}-1}{2i\Delta r}  + (e^{i\theta}+1)\gamma		\right)\frac{3}{i\pi^2}\int_{S^*M} F_\psi\frac{dxdy}{y^2}\\
 & \propto &\frac{1}{2\pi i} \int_{S^*M} F_\psi\frac{dxdy}{y^2}
\end{eqnarray*}
where $\theta$ is such that
$$e^{i\theta} = \frac{\zeta(1+2ir')\zeta(1-2ir'')\Gamma(\frac{1}{2}+ir')\Gamma(\frac{1}{2}-ir'')}{\zeta(1+2ir'')\zeta(1-2ir')\Gamma(\frac{1}{2}+ir'')\Gamma(\frac{1}{2}-ir')}$$
as above, and in fact
$$\int_{S^*M} F_\psi d\mu_{r', r''} \sim \left(\frac{e^{i\theta}-1}{2i\Delta r}  + (e^{i\theta}+1) \gamma		\right)\frac{6}{\pi}\int_{S^*M} F_\psi\frac{dxdy}{y^2}$$

Moreover, if $\Delta r =o(\log\log{r_j} /\log{r_j})$, then we have $\theta = 2i\Delta r \log{r_j}+o(1)$, 
in which case the constant $\gamma$-term is negligible and we get the final formula
$$\int_{S^*M} F_\psi d\mu_{r', r''} \sim \frac{(e^{2i\log{r_j}\cdot \Delta r}-1)}{2 i\Delta r}\frac{6}{\pi} \int_{S^*M} F_\psi\frac{dxdy}{y^2}$$
appearing in Theorem~\ref{sharp asymptotic}.

\subsection{Weight $k\neq 0$ Incomplete Eisenstein Series --- Estimating the Error}\label{jakobson}

As in  \cite{Jak}, we finish the proofs of Theorems~\ref{equidist} and \ref{sharp asymptotic} by evaluating $\int_{S^*M} F_\psi d\mu_{r',r''}$  where
\begin{eqnarray}
F_\psi(z,\theta) & = & \frac{1}{2\pi i} e^{2ik\theta} \int_{Re(s)=\frac{3}{2}} L_\psi(s) E_{2k}(z,s) ds\label{incomplete eisen}\\
L_\psi(s) & = &\int_{y=0}^\infty \psi(y) y^{-s}\frac{dy}{y}\nonumber
\end{eqnarray}
is an incomplete Eisenstein series of weight $2k$.  As before, we  consider 
$$\int_{\Gamma\backslash\mathbb{H}} F_\psi(z,0)E(z, \frac{1}{2}-ir')E_{-2k}(z, \frac{1}{2}+ir'')\frac{dxdy}{y^2}$$
after integrating out $\theta$.  
We then substitute (\ref{incomplete eisen}) for $F_\psi$, and after ``unfolding" the $E_{2k}$'s in the expansion (\ref{incomplete eisen}) we get
\begin{eqnarray*}
\int F_\psi d\mu_{r', r''} & = & \frac{1}{2\pi i}\int_{Re(s)=\frac{3}{2}} \int_{y=0}^\infty \int_{x=0}^1 L_\psi(s) y^s E(z, 1/2-ir') E_{-2k}(z,1/2+ir'') \frac{dxdy}{y^2}
\end{eqnarray*}
After replacing $E(z,1/2-ir')$ and $E_{-2k}(z, 1/2+ir'')$ with their Fourier expansions (\ref{eisen 0 expansion}) and (\ref{eisen -2k expansion},\ref{eisen 2k expansion}), we 
 substitute \cite[9.235.2]{G-R}
$$W_{0,-ir'}(u) = \sqrt{\frac{u}{\pi}}K_{-ir'}(u/2)		$$
in the expansion of $E(z,\frac{1}{2}-ir')$, and after integrating out $x$ we get
\begin{eqnarray*}
\lefteqn{\int F_\psi d\mu_{r', r''}}\\ 
& = & \frac{1}{2\pi i } \int_{Re(s)=\frac{3}{2}}\int_{0}^\infty y^sL_\psi(s) \frac{dy}{y}ds\cdot \left( 1 + \phi(1/2-ir')\frac{(-1)^k\Gamma^2(\frac{1}{2}+ir'')\phi(1/2+ir'')}{\Gamma(\frac{1}{2}-k+ir'')\Gamma(\frac{1}{2}+k+ir'')}\right)\\ 
& & + \Big(\text{rapidly decreasing in } r_j\Big)\\
& & + \frac{(-1)^k\Gamma(\frac{1}{2}+ir'')}{4\xi(1-2ir')\xi(1+2ir'')}\frac{1}{2\pi i}\int_{Re(s)=\frac{3}{2}} \int_0^\infty y^sL_\psi(s) \sum_{n=1}^\infty \frac{\sigma_{2ir'}(n)\sigma_{-2ir''}(n)}{n^{1+i\Delta r}}\times\\
& & \times W_{0,-ir'}(4\pi ny)\left(	\frac{W_{k, ir''}(4\pi n y)}{\Gamma(\frac{1}{2}+k+ir'')} + \frac{W_{-k, ir''}(4\pi n y)}{\Gamma(\frac{1}{2}-k+ir'')}	\right) \frac{dy}{y^2}ds
\end{eqnarray*}
where the ``rapidly decreasing in $r_j$" term comes from integrating the highly oscillatory cross-terms 
$$\int_{Re(s)=2}\int_0^\infty y^{s-1\pm i(r'+r'')}L_\psi(s)dyds$$
as in section~\ref{luo sarnak}, and similarly by the Mellin inversion formula 
$$\psi(y) = \frac{1}{2\pi i } \int_{Re(s)=\frac{3}{2}} L_\psi(s) y^s ds$$
together with the fact that for any $r\in\mathbb{R}$ we have
\begin{eqnarray*}
\left| \phi(\frac{1}{2}+ir)\right|  & = & \left| \frac{\xi(2ir)}{\xi(1+2ir)}\right| = \left| \frac{\xi(1-2ir)}{\xi(1+2ir)}\right| =1 
\end{eqnarray*}
by the functional equation $\xi(s) = \xi(1-s)$ and $|\xi(\bar{s})| = |{\xi(s)}|$, and together with the fact that the factor
$$	\frac{(-1)^k\Gamma^2(\frac{1}{2}+ir'')}{\Gamma(\frac{1}{2}-k+ir'') \Gamma(\frac{1}{2}+k+ir'')}	$$
is bounded by Stirling's approximation $|\Gamma(\sigma+ir)| \sim \sqrt{2\pi} e^{-\pi|r|/2}|r|^{\sigma-\frac{1}{2}}$;
we see that the first term, giving the incoming and outgoing contributions
\begin{eqnarray*}
\left|	 \int_{Re(s) = \frac{3}{2}}\int_{0}^\infty y^sL_\psi(s) \frac{dy}{y}ds\cdot \Big( 1 + \frac{(-1)^k\Gamma^2(\frac{1}{2}+ir'')\phi(\frac{1}{2}-ir')\phi(\frac{1}{2}+ir'')}{\Gamma(\frac{1}{2}-k+ir'')\Gamma(\frac{1}{2}+k+ir'')}\Big)\right|	\\
\lesssim  \int_0^\infty \psi(y) \frac{dy}{y}
\end{eqnarray*}
is again uniformly bounded in $r_j$.

Therefore, after interchanging summation and integration and changing variables to $u=4\pi ny$, we get
\begin{eqnarray}
\lefteqn{\int F_\psi d\mu_{r', r''}  + O(1)}\nonumber\\
& = & \frac{(-1)^k\Gamma(\frac{1}{2}+ir'')}{4\xi(1-2ir')\xi(1+2ir'')}\frac{1}{2\pi i} \int_{Re(s)=\frac{3}{2}}(4\pi)^{1-s}L_\psi(s) \left(	\sum_{n=1}^\infty \frac{\sigma_{2ir'}(n)\sigma_{-2ir''}(n)}{n^{s+i\Delta r}}\right)\times\nonumber\\
& & \times \left(	\frac{\int_0^\infty W_{k,ir''}(u) W_{0,-ir'}(u) u^s\frac{du}{u^2}}{\Gamma(\frac{1}{2}+k+ir'')}  + \frac{ \int_0^\infty W_{-k,ir''}(u) W_{0,-ir'}(u) u^s\frac{du}{u^2}}{\Gamma(\frac{1}{2}-k+ir'')}  \right) ds\label{incomplete eisen integrals}
\end{eqnarray}

Once again the infinite sum in (\ref{incomplete eisen integrals}) can be evaluated, as was first done by Ramanujan \cite{Ramanujan_sum_formulas}
$$\sum_{n=1}^\infty \frac{\sigma_{2ir'}(n)\sigma_{-2ir''}(n)}{n^{s+i\Delta r}} = \frac{\zeta(s+i\Delta r) \zeta(s-i\Sigma r)\zeta(s+i\Sigma r)\zeta(s-i\Delta r)}{\zeta(2s)}$$
Denote the two integrals (over $u$) in (\ref{incomplete eisen integrals}) by $I^3_k$ and $I^4_k$; they can also be evaluated for $Re(s)=\frac{3}{2}<k+1$ (see \cite[7.611.7]{G-R})
\begin{eqnarray*}
I_k^3  (s) & = & \frac{\Gamma(s-i\Delta r)\Gamma(s+i\Sigma r)\Gamma(-2ir'')}{\Gamma(\frac{1}{2}-k-ir'')\Gamma(s+\frac{1}{2}+ir'')} \times\\ 
 &  & \times \  _3F_2 \left(s-i\Delta r, s+i\Sigma r , \frac{1}{2}-k+ir'' ; 1+2ir'' , s+ \frac{1}{2}+ir'' \right)\\
& & + \frac{\Gamma(s-i\Sigma r)\Gamma(s+i\Delta r)\Gamma(2ir'')}{\Gamma(\frac{1}{2}-k+ir'')\Gamma(s+\frac{1}{2}-ir'')}\times \\ 
 &  & \times \  _3F_2 \left(s-i\Sigma r, s+i\Delta r , \frac{1}{2}-k-ir'' ; 1-2ir'' , s+\frac{1}{2} - ir'' \right)\\
 I_k^4  (s) & = & \frac{\Gamma(s-i\Delta r)\Gamma(s-i\Sigma r)\Gamma(2ir')}{\Gamma(\frac{1}{2}+ir')\Gamma(s+\frac{1}{2}+k-ir')} \times\\ 
 &  & \times \  _3F_2 \left(s-i\Delta r, s-i\Sigma r , \frac{1}{2}-ir' ; 1-2ir' , s+\frac{1}{2}+k-ir' \right)\\
& & + \frac{\Gamma(s+i\Sigma r)\Gamma(s+i\Delta r)\Gamma(-2ir')}{\Gamma(\frac{1}{2}-ir')\Gamma(s+\frac{1}{2}+k+ir')} \times\\ 
 &  & \times \  _3F_2 \left(s+i\Sigma r, s+i\Delta r , \frac{1}{2}+ir' ; 1+2ir' , s+\frac{1}{2}+k +ir' \right)
\end{eqnarray*}

As in \cite{Jak}, we will transform these into more convenient form.  Starting with $I_3^k$, we use the formula \cite[\S 3.5, p.18]{Bailey} (see also \cite[(2.9)]{Jak})
$$\ _3F_2 (a,b,c; e,f) = \frac{\Gamma(e+f-a-b-c)\Gamma(f)}{\Gamma(f-a)\Gamma(e+f-b-c)}\ _3F_2 (a, e-b, e-c; e, e+f-b-c)$$
with $a=s\mp i\Delta r, b=s\pm i \Sigma r, c= \frac{1}{2}-k\pm ir'', e=1\pm 2ir'', f=s+\frac{1}{2}\pm ir''$ to get
\begin{eqnarray}\label{Ik3}
I_k^3  (s) & = & \frac{\Gamma(s-i\Delta r)\Gamma(s+i\Sigma r)\Gamma(-2ir'')\Gamma(k+1-s)}{\Gamma(\frac{1}{2}-k-ir'')\Gamma(\frac{1}{2}+ir')\Gamma(k+1-i\Delta r)} \times\\ 
 &  & \times \  _3F_2 \left(s-i\Delta r, 1-s-i\Delta r , \frac{1}{2}+k+ir'' ; 1+2ir'' , 1+k-i\Delta r \right)\nonumber\\
& & + \frac{\Gamma(s-i\Sigma r)\Gamma(s+i\Delta r)\Gamma(2ir'')\Gamma(k+1-s)}{\Gamma(\frac{1}{2}-k+ir'')\Gamma(\frac{1}{2}-ir')\Gamma(k+1+i\Delta r)}\times \nonumber\\ 
 &  & \times \  _3F_2 \left(s+i\Delta r, 1-s+i\Delta r , \frac{1}{2}+k-ir'' ; 1-2ir'' , 1+k +i\Delta r \right)\nonumber
\end{eqnarray}

We transform $I_k^4$ via the formula \cite[\S 3.5, p.18]{Bailey} (see also\footnote{There is a typo in the formula in \cite{Jak}, with the gamma factors appearing on the wrong side of the equation; but the correct formula is applied to (3.7b) there.} \cite[(3.7b)]{Jak})
\begin{eqnarray*}
\lefteqn{\frac{\Gamma(f-a)\Gamma(2+b+c-e-f)}{\Gamma(1-a)\Gamma(1+b+c-e)} 
\ _3F_2 (1+a+b+c-e-f, b, c; 1+b+c-e, 1+b+c-f)}\\
& = & \ _3F_2 (1+a+b+c-e-f, 1-f+b, 1-f+c; 2+b+c-e-f, 1+b+c-f)
\end{eqnarray*}
setting $a=s-k, b=s\mp i\Sigma r, c=\frac{1}{2} \mp ir', 1-e = k\pm i\Sigma r, 1-f=\frac{1}{2}-s\pm ir''$ to get
\begin{eqnarray}\label{Ik4}
I_k^4 (s) & = & \frac{\Gamma(s-i\Delta r)\Gamma(s-i\Sigma r)\Gamma(2ir')\Gamma(1-s+k)}{\Gamma(\frac{1}{2}+ir')\Gamma(\frac{1}{2}+k-ir'')\Gamma(1+k-i\Delta r)} \times\\ 
 &  & \times \  _3F_2 \left(s-i\Delta r, \frac{1}{2}-ir' , 1-s-i\Delta r ;  1+k-i\Delta r, 1-2ir' \right)\nonumber\\
& & + \frac{\Gamma(s+i\Sigma r)\Gamma(s+i\Delta r)\Gamma(-2ir')\Gamma(1-s+k)}{\Gamma(\frac{1}{2}-ir')\Gamma(\frac{1}{2}+k+ir'')\Gamma(1+k+i\Delta r)} \times\nonumber\\ 
 &  & \times \  _3F_2 \left(s+i\Delta r, \frac{1}{2}+ir' , 1-s+i\Delta r ; 1+2ir' , 1+k+i\Delta r \right)\nonumber
\end{eqnarray}

These transformations are applied in part so that each of the hypergeometric functions $_3F_2$ converge for arbitrary $s,r',r''$.  We shift the line of integration in (\ref{incomplete eisen integrals}) 
to $Re(s)=\frac{1}{2}$, where the integral is estimated as in \cite{Jak} to decay as $O_{k,\psi, \epsilon,\omega}(|r_j|^{-1/6+\epsilon})$, using Stirling's formula to estimate the gamma factors, estimates on the growth of our hypergeometric functions $_3F_2$ on the line $Re(s)=\frac{1}{2}$ (see \cite[Appendix]{Jak}), and Weyl's subconvexity bound $|\zeta(\frac{1}{2}+ir)|=O_\epsilon(|r|^{1/6+\epsilon})$.  Note that the extra $\Delta r$ terms sprinkled in do not affect the asymptotics of the $_3F_2$ terms or the $|\zeta|$ factors, or the absolute value $|\Gamma|$ factors by Lemma~\ref{Gamma}, so that the overall estimate is not affected.  The ``top" and ``bottom" integrals over $\frac{1}{2} \leq Re(s) \leq \frac{3}{2}$ for  large $|Im(s)|$ are similarly estimated as in \cite{Jak} to decay as well.  

In shifting the line of integration, we pass through $4$ (simple) poles coming from the $\zeta$ factors at argument $1$.  The poles at $s=1\pm i\Sigma r$ are readily seen to decay rapidly in $r_j$, since $L_\psi$ is rapidly decreasing (again estimating all relevant $|\Gamma|$ and $|\zeta|$ factors as in \cite{Jak}).  Thus it remains to examine the residues at $s=1\pm i\Delta r$, which will give the main term.  Write
\begin{eqnarray*}
B_k(s) & =  & \frac{(-1)^k(4\pi)^{1-s}\Gamma(\frac{1}{2}+ir'')}{4\xi(1-2ir')\xi(1+2ir'')}L_\psi(s) \frac{ \zeta(s-i\Sigma r)\zeta(s+i\Sigma r)}{\zeta(2s)} \times\nonumber\\
& & \times \left(	\frac{I_k^3(s)}{\Gamma(\frac{1}{2}+k+ir'')}  + \frac{ I_k^4(s)}{\Gamma(\frac{1}{2}-k+ir'')}  \right) 
\end{eqnarray*}
so that the integrand (in $s$) of (\ref{incomplete eisen integrals}) becomes
$$\zeta(s+i\Delta r) \zeta(s-i\Delta r) B_k(s)$$
where $B_k$ is holomorphic on a neighborhood of $s=1$, and it is the residue of this integrand at the poles $s=1\pm i\Delta r$ that again determine the asymptotics of $\int F_\psi d\mu_{r', r''}$.

We wish to compare the residues here with the corresponding ones in section~\ref{luo sarnak} for the weight $0$ component of the measure.  Comparing like terms, we see that
\begin{eqnarray*}
B_k(s) &= & \frac{(-1)^k4^{1-s}\pi}{2} B_0(s) \frac{\Gamma(\frac{1}{2}+ir'')\Gamma(s)}{\Gamma(\frac{s-i\Delta r}{2})
\Gamma(\frac{s-i\Sigma r}{2})\Gamma(\frac{s+i\Sigma r}{2})\Gamma(\frac{s+i\Delta r}{2})} \times\\
& & \times  \left(	\frac{I_k^3(s)}{\Gamma(\frac{1}{2}+k+ir'')}  + \frac{ I_k^4(s)}{\Gamma(\frac{1}{2}-k+ir'')}  \right)
\end{eqnarray*}
and so
\begin{eqnarray}
\frac{B_k(1-i\Delta r)}{B_0(1-i\Delta r)} & \sim & \frac{(-1)^k}{2}  \frac{1}{\Gamma(\frac{1}{2}-ir')}\left(	\frac{I_k^3(1-i\Delta r)}{\Gamma(\frac{1}{2}+k+ir'')}  + \frac{ I_k^4(1-i\Delta r)}{\Gamma(\frac{1}{2}-k+ir'')}  \right)\label{B_k B_0 -}\\
\frac{B_k(1+i\Delta r)}{B_0(1+i\Delta r)} & \sim & \frac{(-1)^k}{2}  \frac{\Gamma(\frac{1}{2}+ir'')}
{\Gamma(\frac{1}{2}-ir'')\Gamma(\frac{1}{2}+ir')}
   \left(	\frac{I_k^3(1+i\Delta r)}{\Gamma(\frac{1}{2}+k+ir'')}  + \frac{ I_k^4(1+i\Delta r)}{\Gamma(\frac{1}{2}-k+ir'')}  \right)\label{B_k B_0 +}
\end{eqnarray}

The analog of Lemma~\ref{|B_0|} for weight $k\neq 0$ is:
\begin{Lemma}\label{|B_k|}
We have
$$B_k(1+i\Delta r) \sim \frac{\xi(1+2ir')\xi(1-2ir'')}{\xi(1+2ir'')\xi(1-2ir')}B_k(1-i\Delta r)$$
In particular we have $|B_k(1+i\Delta r)| \sim |B_k(1-i\Delta r)|$,
and if $|\Delta r| \lesssim 1/\log{r_j}$, then 
$$B_k(1+i\Delta r) \sim e^{2i\Delta r\log{r_j}} B_k(1-i\Delta r)$$
\end{Lemma}

{\em Proof:}  Stirling's formula implies that as long as $\Delta r \to 0$ we have 
$$ \frac{I_k^3(1+i\Delta r)}{I_k^3(1-i\Delta r)} \sim \frac{\Gamma(1+2ir')}{\Gamma(1+2ir'')} \sim \frac{I_k^4(1+i\Delta r)}{I_k^4(1-i\Delta r)}$$
since the only terms in (\ref{Ik3},\ref{Ik4}) contributing non-trivially to the asymptotic ratio between the values at $s=1\pm i\Delta r$  are the terms of the form $\Gamma(s\pm i\Sigma r)$--- all other terms are either independent of $s$ or have bounded arguments, so that the limit $\Delta r\to 0$ gives the same asymptotic for $s=1+i\Delta r$ and $s=1-i\Delta r$; but for these terms we have
$$\frac{\Gamma(1+i\Delta r +i\Sigma r)}{\Gamma(1-i\Delta r +i\Sigma r)}=\frac{\Gamma(1+2ir')}{\Gamma(1+2ir'')}  \sim \frac{\Gamma(1-2ir'')}{\Gamma(1-2ir')}=\frac{\Gamma(1+i\Delta r -i\Sigma r)}{\Gamma(1-i\Delta r -i\Sigma r)}$$
by Stirling's approximation.  Moreover,
$$  \frac{\Gamma(\frac{1}{2}+ir'')}{\Gamma(\frac{1}{2}-ir'')\Gamma(\frac{1}{2}+ir')}
\frac{\Gamma(1+2ir')}{\Gamma(1+2ir'')}
\sim  \frac{1}{\Gamma(\frac{1}{2}-ir')} $$
so that the asymptotics (\ref{B_k B_0 -},\ref{B_k B_0 +}) give 
$$  \frac{B_k(1+i\Delta r)}{B_0(1+i\Delta r)}\sim \frac{B_k(1-i\Delta r)}{B_0(1-i\Delta r)}  $$
and we are reduced to Lemma~\ref{|B_0|}.  $\Box$

To complete the proof of Theorem~\ref{equidist} it now remains to show that for every $k\neq0$
$$  B_k(1-i\Delta r) = o(B_0(1-i\Delta r))    $$
which will follow from
\begin{equation}\label{weight k weight 0}
\frac{1}{\Gamma(\frac{1}{2}-ir')} \left(    \frac{I_k^3(1-i\Delta r)}{\Gamma(\frac{1}{2}+k+ir'')}+\frac{I_k^4(1-i\Delta r)}{\Gamma(\frac{1}{2}-k+ir'')}    \right) = o(1)
\end{equation}

To prove (\ref{weight k weight 0}), we first recall from \cite{Jak} that all of the hypergeometric series factors in (\ref{Ik3}) and (\ref{Ik4}) converge uniformly in $r', r''$ near $s=1$, and so we may take the limit $\Delta r\to 0$ term-by-term.  Recall that the hypergeometric series $_3F_2$ is defined by
$$ \ _3 F_2(a,b,c; e,f) = 1 + \sum_{n=1}^\infty \frac{(a)_n(b)_n(c)_n}{(e)_n(f)_n n!}$$
where $(z)_n = z(z+1)(z+2)\ldots (z+n-1)$.  Since the term $1-s\pm i\Delta r$ appears in the numerator of each $\ _3F_2$, and we have $(1-s\pm i\Delta r)_n\to 0$ for $s=1\pm i\Delta r$ as $\Delta r\to 0$, we find that each $_3F_2$ in (\ref{Ik3}) and (\ref{Ik4}) converges to $1$ as $\Delta r\to 0$.  Hence, we omit these terms in the asymptotic calculation.  Similarly, the terms $\Gamma(s\pm i\Delta r)\sim 1$, and $\Gamma(k+1-s)\sim \frac{1}{k}\Gamma(k+1 \pm i\Delta r)$ when $s=1\pm i\Delta r$, as $\Delta r\to 0$.  Thus, we get
\begin{eqnarray}
\frac{I_k^3(1-i\Delta r)}{\Gamma(\frac{1}{2}-ir')\Gamma(\frac{1}{2}+k+ir'')} 
& \sim & \frac{\Gamma(1+2ir'')\Gamma(-2ir'')}{k\Gamma(\frac{1}{2}-k-ir'')|\Gamma(\frac{1}{2}+ir')|^2\Gamma(\frac{1}{2}+k+ir'')}\label{Ik3_1}\\
& &  + \frac{\Gamma(1-2ir')\Gamma(2ir'')}{k\Gamma(\frac{1}{2}-k+ir'')\Gamma(\frac{1}{2}-ir')^2\Gamma(\frac{1}{2}+k+ir'')}\label{Ik3_2}\\
\frac{I_k^4(1-i\Delta r)}{\Gamma(\frac{1}{2}-ir')\Gamma(\frac{1}{2}-k+ir'')} 
& \sim & \frac{\Gamma(1-2ir')\Gamma(2ir')}{k|\Gamma(\frac{1}{2}+ir')|^2\Gamma(\frac{1}{2}+k-ir'')\Gamma(\frac{1}{2}-k+ir'')}\label{Ik4_1}\\
& & +\frac{\Gamma(1+2ir'')\Gamma(-2ir')}{k\Gamma(\frac{1}{2}-ir')^2\Gamma(\frac{1}{2}+k+ir'')\Gamma(\frac{1}{2}-k+ir'')}\label{Ik4_2}
\end{eqnarray}

A quick estimate using Stirling's formula shows that each term above is bounded, and that their absolute values are asymptotic to each other, but this is not sufficient for us--- we wish to show that there is cancellation in their sum.  Taking first (\ref{Ik3_2}) and (\ref{Ik4_2}), we use the identity $\Gamma(1+s) = s\Gamma(s)$ and 
collect like terms to arrive at
\begin{eqnarray*}
 \lefteqn{\frac{\Gamma(1-2ir')\Gamma(2ir'')}{k\Gamma(\frac{1}{2}-k+ir'')\Gamma(\frac{1}{2}-ir')^2\Gamma(\frac{1}{2}+k+ir'')}
 + \frac{\Gamma(1+2ir'')\Gamma(-2ir')}{k\Gamma(\frac{1}{2}-ir')^2\Gamma(\frac{1}{2}+k+ir'')\Gamma(\frac{1}{2}-k+ir'')}}\\
 & = & \frac{-2ir'\Gamma(-2ir')\Gamma(2ir'') + 2ir''\Gamma(2ir'') \Gamma(-2ir')}
 {k\Gamma(\frac{1}{2}-ir')^2\Gamma(\frac{1}{2}+k+ir'')\Gamma(\frac{1}{2}-k+ir'')}\\
 & = & \frac{\Gamma(-2ir')\Gamma(2ir'')}{k\Gamma(\frac{1}{2}-ir')^2\Gamma(\frac{1}{2}+k+ir'')\Gamma(\frac{1}{2}-k+ir'')}
\cdot (-2i\Delta r) \\
& \lesssim &  \frac{1}{k|r|} \cdot (-2i\Delta r) \to 0
 \end{eqnarray*}
 using Stirling's approximation 
$$\log{\Gamma(z)} = (z-\frac{1}{2})\log{z}-z+ O(1)$$
in the third line.

For the two remaining summands (\ref{Ik3_1}) and (\ref{Ik4_1}), we first observe that
$$  \Gamma(s)\Gamma(1-s) = -s\Gamma(s)\Gamma(-s) = -\Gamma(1+s)\Gamma(-s)$$
which, upon iterating $k$ times, gives
\begin{eqnarray*}
\Gamma(\frac{1}{2}-k-ir'')\Gamma(\frac{1}{2}+k+ir'') & = & (-1)^k \Gamma(\frac{1}{2}-ir'')\Gamma(\frac{1}{2}+ir'')\\
\Gamma(\frac{1}{2}+k-ir'')\Gamma(\frac{1}{2}-k+ir'') & = & (-1)^k \Gamma(\frac{1}{2}-ir'')\Gamma(\frac{1}{2}+ir'')
\end{eqnarray*}
so that
$$\frac{\Gamma(1+2ir'')\Gamma(-2ir'')}{k\Gamma(\frac{1}{2}-k-ir'')|\Gamma(\frac{1}{2}+ir')|^2\Gamma(\frac{1}{2}+k+ir'')} + \frac{\Gamma(1-2ir')\Gamma(2ir')}{k|\Gamma(\frac{1}{2}+ir')|^2\Gamma(\frac{1}{2}+k-ir'')\Gamma(\frac{1}{2}-k+ir'')}$$
\begin{equation}\label{absolute value terms} 
= \frac{(-1)^k \big(2ir''|\Gamma(2ir'')|^2 - 2ir' |\Gamma(2ir')|^2\big)}{k|\Gamma(\frac{1}{2}+ir')|^2|\Gamma(\frac{1}{2}+ir'')|^2}
\end{equation}
But Stirling's approximation  shows that
$$2r''|\Gamma(2ir'')|^2 \sim 2\pi e^{-2\pi|r''|} \sim 2\pi e^{-2\pi|r'|} \sim 2r'|\Gamma(2ir')|^2$$
as $\Delta r\to 0$, and since the denominator 
$$	k|\Gamma(\frac{1}{2}+ir')|^2|\Gamma(\frac{1}{2}+ir'')|^2 \sim 4\pi^2 ke^{-\pi(|r'|+|r''|)} $$
we find that (\ref{absolute value terms}) is asymptotic to
$$\frac{(-1)^ki}{2\pi k} \left(e^{\pi\Delta r} - e^{-\pi\Delta r}\right) \sim \frac{(-1)^ki}{k}  \Delta r \to 0$$
and hence putting the two halves back together, the expression (\ref{weight k weight 0}) tends to $0$ as required.		

Moreover, plugging this into (\ref{B_k B_0 -}) and using Lemma~\ref{|B_k|} we have shown that
$$B_k(1\pm i\Delta r) \sim \frac{1}{2k} i\Delta r \cdot B_0(1\pm i\Delta r)$$
Plugging this into the calculation of residues at the end of section~\ref{luo sarnak}, we conclude that for $F_\psi$ an incomplete Eisenstein series of weight $2k$, we have the asymptotic
$$\int F_\psi d\mu_{r', r''} \sim \frac{3}{2\pi k}(e^{i\theta}-1)\int_{S^*M} F_\psi \frac{dxdy}{y^2}$$
where once again for $\Delta r\lesssim 1/\log{r_j}$ we can evaluate
$$i\theta  = 2i\Delta r \log{r_j} + o(1)$$ 
so that
$$\int F_\psi d\mu_{r', r''} \sim \frac{1}{2k}\Big(e^{2i\log{r_j}\cdot (r'-r'')}-1\Big) \frac{3}{\pi}\int_{S^*M} F_\psi \frac{dxdy}{y^2}$$
and thus when integrating over $r'$ and $r''$   we are left with
\begin{equation}\label{K asymptotic}
\mu_{h_j}(F_\psi) \sim \frac{3}{\pi}\int_{S^*M} F_\psi\frac{dxdy}{y^2} \cdot \frac{1}{2k} \iint h_j(r')\overline{h_j}(r'') \Big[	e^{2i\log{r_j}\cdot (r'-r'')}-1	\Big]dr'dr''
\end{equation}
$\Box$

\vspace{.2in}
{\em Completion of the proof of Theorems~\ref{equidist} and \ref{sharp asymptotic}:}  We have shown that integrals of $d\mu_{r',r''}$ against holomorphic cusp forms, shifted Maass cusp forms, and weight $k\neq 0$ incomplete Eisenstein series all contribute lower order terms to the main asymptotic computed in section~\ref{luo sarnak}, coming from the integral against incomplete Eisenstein series of weight $0$.  Since these span $L^2(S^*M)$, standard approximation arguments (as in \cite{LuoSarnak} and \cite{Jak}) then show that this asymptotic holds for any smooth test function $f\in C^\infty(S^*M)$.

It remains to transfer the asymptotics of $\mu_{r',r''}$ to those of $\mu_{h_j}$, obtained by integrating
$$\mu_{h_j} = \iint h_j(r')\overline{h_j}(r'') \mu_{r',r''} dr'dr''$$
Observe that if $h_j$ satisfies the quasimode condition (\ref{quasimode condition}) with $\omega(r_j) = o(1)$, we can find another sequence $\tilde{\omega}(r_j)\searrow 0$, and a sequence of smooth functions $\tilde{h}_j$ each supported in $[r_j - \tilde{\omega}(r_j), r_j+\tilde{\omega}(r_j)]$ , satisfying $\int_r \tilde{h}_j(r)dr=1$, such that
\begin{equation}\label{h h tilde}
||h_j - \tilde{h}_j||_1 = o(||h_j||_1)=o(1) 
\end{equation}
as $j\to\infty$.  To see this, note that for every fixed $\epsilon$, the quasimode condition implies that
$$\lim_{r_j\to\infty} \int_{|r-r_j|\geq \epsilon}|h_j(r)| dr = o(||h_j||_1)$$
Hence we can find a decreasing sequence of $\{\tilde{\omega}(r_j)\}$ such that 
$$\lim_{r_j\to\infty} \int_{|r-r_j|\geq \tilde{\omega}(r_j)}|h_j(r)| dr = o(||h_j||_1)$$
and we can find  appropriate smooth cutoffs $\tilde{h}_j$ approximating $h_j$ that are supported in $[r_j - \tilde{\omega}(r_j), r_j+\tilde{\omega}(r_j)]$.

We use the simple inequality
$$\mu_{r', r''}(f) \lesssim_f \max\{ \log{r'}, \log{r''}\}$$
which follows from the Cauchy-Schwarz inequality; to apply this we select a smooth, compactly supported, positive function $g$ such that $g\equiv 1$ on the support of $f$, so that we may write $f = g \cdot f$.  Thus \cite{Z1} 
$$Op(f) \sim Op(g) \circ Op(f)\sim Op(g)^* \circ Op(f)$$
and therefore\footnote{The discerning reader will note that the asymptotic $Op(f)\sim Op(g)^*\circ Op(f)$ is valid as operators on $L^2(M)$, while the Eisenstein series are not in $L^2$.  However, the compact support of $f$ (and $g$) mean that the Eisenstein series in the inner product may be replaced with smooth finite truncations away from the cusp, which are then in $L^2$, so that the asymptotic is valid.}
\begin{eqnarray*}
\lefteqn{\left\langle Op(f) E(\cdot, \frac{1}{2}+ir'), E(\cdot, \frac{1}{2}+ir'')\right\rangle}\\
 & \sim & \left\langle Op(f) E(\cdot, \frac{1}{2}+ir'), Op(g)E(\cdot, \frac{1}{2}+ir'')\right\rangle\\
& \lesssim & \mu_{r'}(|f|^2)^{1/2}\mu_{r''}(|g|^2)^{1/2} \\
& \lesssim_{f,g} & (\log{r'})^{1/2}(\log{r''})^{1/2}
\end{eqnarray*}

Since  $\int h_j(r)dr=1$, we then have
\begin{eqnarray*}
\int_{r''} h_j(r'')(\log{r''})^{1/2} dr'' 
& \leq & 	\int_{r''\leq 2r_j} h_j(r'')(\log{r_j})^{1/2}dr'' + \int_{r''> 2r_j} h_j(r'')|r''-r_j|dr''	\\
& \leq &  (\log{r_j})^{1/2} + o(1)\\
& \lesssim &  (\log{r_j})^{1/2}
\end{eqnarray*}
and similarly for $\tilde{h}_j$.
Moreover, since $\tilde{h}_j$ is supported near $r_j$, and $||h_j-\tilde{h}_j||_1=o(1)$, we also have
\begin{eqnarray*}
\lefteqn{\int_{r''} [h_j(r'')-\tilde{h}_j(r'')] (\log{r''})^{1/2} dr''}\\ 
& \lesssim & 	\int_{r''\leq 2r_j} [h_j(r'')-\tilde{h}_j(r'')](\log{r_j})^{1/2}dr'' + \int_{r''> 2r_j} h_j(r'')|r''-r_j|dr''	\\
& = &  o(\sqrt{\log{r_j}}) + o(1)\\
& = & o(\sqrt{\log{r_j}})
\end{eqnarray*}
Putting these together gives
\begin{eqnarray*}
\lefteqn{\mu_{h_j}(f) - \mu_{\tilde{h}_j} (f)}\\
& = & \iint h_j(r')\overline{h_j}(r'')\mu_{r',r''}(f)dr'dr'' - \iint \tilde{h}_j(r')\overline{\tilde{h}_j}(r'')\mu_{r',r''}(f)dr'dr''\\
& = & \iint h_j(r')\overline{[h_j(r'')-\tilde{h}_j(r'')]}\mu_{r',r''}(f)dr'dr'' + \iint [h_j(r')-\tilde{h}_j(r')]\overline{\tilde{h}_j}(r'')\mu_{r', r''}(f)dr'dr''\\
& \lesssim & \int_{r'} h_j(r') (\log{r'})^{1/2}\cdot o(\sqrt{\log{r_j}})dr' + \int_{r'} [h_j(r')-\tilde{h}_j(r')] (\log{r'})^{1/2}(\log{r_j})^{1/2}	dr'\\
& = & o(\log{r_j})
\end{eqnarray*}
Thus $\frac{1}{2\log{r_j}}\mu_{h_j}$ and $\frac{1}{2\log{r_j}}\mu_{\tilde{h}_j}$ have the same weak-* limit points, and completes the proof of Theorem~\ref{equidist}.

The above observations apply equally well to the case where $\int_r h(r)|r-r_j|dr =o(\log\log{r_j}/\log{r_j})$--- that is, we can similarly find a sequence $\{\tilde{h}_j\}$ with each $\tilde{h}_j$ supported in a window of size $\tilde{\omega}(r_j)=o(\log\log{r_j}/\log{r_j})$, such that 
$||h_j-\tilde{h}_j||_1 = o(1)$, and thus the microlocal lifts $\mu_{h_j}$ and $\mu_{\tilde{h}_j}$ are asymptotic.  Moreover,   the convolutions 
$$ \iint h_j(r')\overline{h_j}(r'') 
\frac{e^{2i\log{r_j}\cdot(r'-r'')}-1}{2i\log{r_j}\cdot(r'-r'')} 
dr'dr'' 
\sim  \iint \tilde{h}_j(r')\overline{\tilde{h}_j}(r'') 
\frac{e^{2i\log{r_j}\cdot(r'-r'')}-1}{2i\log{r_j}\cdot(r'-r'')} 
dr'dr''$$
by the same argument.
Thus the asymptotic expression in the second part of Theorem~\ref{sharp asymptotic}--- which clearly holds for $\tilde{h}_j$, since it is supported in a window of width $o(\log\log{r_j}/\log{r_j})$, and therefore we can simply plug in the expression from the first part of the Theorem--- holds for $\mu_{h_j}$ as well.  Then, since 
$$\lim_{r\to\infty} \frac{e^{2i\log{r_j}\cdot(r'-r'')}-1}{2i\log{r_j}\cdot (r'-r'')]} = 1$$
when $|r'-r''| = o(1/\log{r_j})$, the final statement follows and completes the proof of Theorem~\ref{sharp asymptotic}.  $\Box$

\section{QUE Thresholds}\label{threshold}

In this section, we interpret the results of Theorems~\ref{equidist} and \ref{sharp asymptotic}, in the context of Conjecture~\ref{hypconj} and the following result for compact surfaces:
\begin{Theorem}[\cite{locquasi}]\label{locquasi}
Let $M=\Gamma\backslash\mathbb{H}$ be a compact hyperbolic surface, and $\gamma\subset S^*M$ a closed geodesic.  Then for any $\epsilon>0$, there exists $\delta>0$ and a sequence of $\left(\frac{\epsilon}{\log{r}}\right)$-quasimodes on $M$ whose microlocal lifts do not equidistribute, and in fact concentrate mass $\geq \delta$ on the geodesic $\gamma$.
\end{Theorem}
The main idea in the construction is to take a spherical $1$-quasimode centered at a point on $\gamma$, and then average this over a piece of stable horocycle by convolving with a fixed, smooth, compactly supported test function along the stable direction.  This averaging over a stable horocycle introduces constructive interferences along the $\gamma$ direction, and destructive interferences away from $\gamma$.  One then applies the wave equation up to time $T=c\log{r}$, and averages over propagation times; so that on the spectral side we improve the quality of the quasimodes to $O(1/\log{r})$, but on the other hand the constructive interferences along $\gamma$ preserve positive measure near the geodesic, while the destructive interferences prevent the waves from coming back to $\gamma$ to interfere.

A natural question in light of Theorem~\ref{equidist} is to understand why the situation is so different for Eisenstein series.  First and foremost, the construction of \cite{locquasi} fails for Eisenstein series because they furnish a meager part of the spectrum, as first observed by Selberg (see eg. \cite[11.1]{Iwaniec}).   One can assume that the convolution of a spherical $1$-quasimode along a piece of stable non-periodic horocycle is primarily supported on the discrete spectrum, and therefore offers no information on the asymptotics of functions supported in the Eisenstein spectrum.  The exception to this is if $\gamma$ is taken to be a cusp-bound geodesic--- in this case one can replace convolution with a test function, with a complete average over the periodic horocycle.  Since the discrete spectrum is orthogonal to all functions invariant along periodic horocycles, this process yields a quasimode supported in the continuous spectrum.  However, mass localizing on cusp-bound geodesics escapes quickly from compact sets and thus localizes its mass in the cusp, and therefore does not contribute to the asymptotics of the $\mu_{h_j}(f)$.  

There is also a good reason for the equidistribution of $o(1)$-quasimodes, due to the intrinsic arithmetic structure of the Eisenstein series.  Recall that any Eisenstein series $E(z,\frac{1}{2}+ir)$ is also an eigenfunction of all Hecke operators $T_p$, with eigenvalue $1+p^{-2ir}$ depending {\em continuously} on the Laplace spectral parameter $r$.  This means that any Eisenstein $o(1)$-quasimode is automatically also an $o(1)$-quasimode for $T_p$.  This is a highly unusual situation, and surely does not hold for cusp forms.

In joint work with E. Lindenstrauss \cite{jointQmodes}, we showed that on a compact congruence surface any quantum limit arising from joint $o(1)$-quasimodes for $\Delta$ and one $T_p$ (outside a finite set of bad primes depending on the surface), must be the uniform Liouville measure.  The argument is based on the measure rigidity results of \cite{Lin}; and in particular, can only show that the limit measure is proportional to Liouville measure--- in the case of a compact surface, any limit measure is automatically a probability measure, so this does not impact the result.  Here, however, it is interesting to note that the arithmetic structure responsible for Theroem~\ref{equidist} cannot control the scaling of the measure, and in fact one can have $\mu=0$ for sequences of sufficiently weak quasimodes.  

In this context, it is informative to look at the $o(1/\log{r})$ threshold of Conjecture~\ref{hypconj}, where in light of Theorem~\ref{locquasi} we would have expected to see scarring on a geodesic for $\left(\frac{\epsilon}{\log{r}}\right)$-quasimodes.  Though Theorem~\ref{equidist} shows that this cannot happen on a closed geodesic, we can still identify extra concentration near cusp-bound geodesics from the lower order asymptotics of (\ref{K asymptotic}).  Note that for $o(1/\log{r})$ quasimodes, this asymptotic of the $K$-Fourier coefficients vanishes; while for weaker $\left(\epsilon/\log{r}\right)$-quasimodes, we get $K$-Fourier coefficients $\gtrsim 1/k$, with the implied proportionality constant growing linearly in $\epsilon$.  It is in this sense that we observe extra localization near cusp-bound geodesics, which are invisible in the main asymptotic, but show up as lower order terms, and may be considered responsible for the diminished mass of the quantum limits and the extra concentration in the cusp.

To further develop this idea, we note that the Fourier transform of the convolution kernel in (\ref{K asymptotic})
$$c_{r_j}(s) := e^{-2is\log{r_j}}-1$$
is given by $\widehat{c_{r_j}}(t) = \delta_{-2\log{r_j}} - \delta_0$.  Thus these $K$-Fourier coefficients of $\mu_{h_j}$ are proportional to the asymptotic
\begin{eqnarray*}
 \iint h_j(r')\overline{h_j}(r'') \overline{c_{r_j}}(r'-r'')dr'dr''
& \sim & \langle h_j, [h_j\ast c_{r_j}]\rangle\\
& \sim & \left\langle \widehat{h_j},  \widehat{h_j\ast c_{r_j}}\right\rangle \\
& \sim &  |\widehat{h_j}(-2\log{r_j})|^2- |\widehat{h_j}(0)|^2\\
& \sim &   |\widehat{h_j}(-2\log{r_j})|^2 -1 
\end{eqnarray*}
Since $\widehat{h_j}(t)$ gives propagation times for plane waves, we interpret this as saying that the Eisenstein plane waves equidistribute up to time\footnote{Here the sign of $r$ is chosen positive if the Eisenstein plane waves are pointing ``out" towards the cusp, and thus equidistribution in compact sets comes from backward time evolution to $-2\log{r_j}$.  The $\delta_0$ term comes from the outgoing wave at time $0$.} $2\log{r_j} =  2T_E$, and then immediately afterwards localize extra mass near cusp-bound geodesics.  These contribute lower-order asymptotics at time $ 2T_E$, after which they escape to the cusp and concentrate extra mass there, decreasing the mass left behind in compact sets due to destructive interference away from the cusp-bound geodesics, in analogy with the arguments of \cite{locquasi}.  It is also worth noting that the standard semiclassical analysis only shows that wave propagation agrees with the geodesic flow up to the Ehrenfest time $T_E$, and here we see agreement (i.e., that they both equidistribute) to twice this time.  An observation of Sarnak \cite{SarnakLetterRudnick}, allowing one to halve the multiplicity bound for the modular surface, is likely related, and perhaps one could use this to show semiclassical agreement up to $2T_E$ in this case.

As a final remark, we note that the dearth of continuous spectrum on congruence surfaces is expected to be a very special feature of the arithmetic structure, and that generically the discrete spectrum should be small and the continuous spectrum dominant \cite{PhillipsSarnak}.  We expect therefore that generically, the continuous spectrum would follow the behavior observed in \cite{locquasi} for compact surfaces, and not have equidistribution of weak $o(1)$-quasimodes as in Theorem~\ref{equidist}.
Since we have virtually no tools to analyze spectral data on generic surfaces, it would be difficult to prove any statements; on the other hand, it might be interesting to look at numerics for, say, non-arithmetic Hecke triangle groups, and see if concentration on closed geodesics can be observed for Eisenstein $O(1/\log{r})$-quasimodes.

\subsection*{Acknowledgment}
The author is indebted to Alex Kontorovich for patiently explaining the ideas behind the work of Luo-Sarnak \cite{LuoSarnak} and Jakobson \cite{Jak}.  We would also like to thank Elon Lindenstrauss  for helpful discussions and suggestions.

\def\cprime{$'$}

\end{document}